\documentclass[a4paper,english,sumlimits,reqno]{amsart}

\usepackage[T1]{fontenc}
\usepackage[utf8]{inputenc}
\usepackage{lmodern}

\usepackage{babel}

\usepackage{xargs}

\usepackage{amssymb}
\usepackage{amsmath}
\usepackage{amsthm}
\usepackage{mathtools}
\usepackage{bbm}
\usepackage{bm}
\usepackage{amsbsy}

\usepackage[mathcal]{eucal}
\usepackage{mathrsfs}

\usepackage{stmaryrd}

\usepackage[all]{xy}
\newcommand{\vcxymatrix}[1]{\vcenter{\xymatrix{#1}}}

\usepackage{fp}

\usepackage[section]{placeins}
\usepackage{float}

\usepackage{enumerate}

\newcounter{proof}
{\stepcounter{proof}\begin{proof}}%
{\end{proof}}%
\newcounter{proofstep}[proof]
{\refstepcounter{proofstep}\bigskip\par\noindent%
  \ifthenelse{\isempty{#1}}
    {\textbf{Step \theproofstep.}}
    {\textbf{#1.}}
  \noindent}%
{\par}%
\newcounter{proofcase}[proof]
{\refstepcounter{proofcase}\bigskip\par\noindent%
  \ifthenelse{\isempty{#1}}
    {\textbf{Case \theproofcase.}}
    {\textbf{#1.}}
  \noindent}%
{\par}%

\usepackage{hyperref}

\theoremstyle{plain}
\newtheorem{thm}{Theorem}[section]
\newtheorem*{thm*}{Theorem}

\newtheorem{lem}[thm]{Lemma}

\theoremstyle{definition}

\theoremstyle{remark}

\newtheorem{rem}[thm]{Remark}

\numberwithin{equation}{section}

\newcommandx{\textref}[2][1=]{(\hyperref[#2]{#1\ref*{#2}})}

\DeclareMathOperator*{\Union}{\bigcup}

\newcommand{\charfun}{\scalebox{1.0}{\ensuremath{\mathbbm 1}}}

\newcommand{\dint}{\mathscr D}

\newcommand{\dif}{\ensuremath{\, \mathrm d}}

\DeclareMathOperator{\Id}{Id}

\DeclareMathOperator{\spn}{span}

\DeclareMathOperator{\cond}{\mathbb E}

\DeclareMathOperator{\sqfun}{\mathbb S}

\newcommand{\diindex}{\ensuremath{\mathcal O}}

\newcommand{\bmo}{\ensuremath{\mathrm{BMO}}}

\begin{document}

\title[Factorization in $SL^\infty$]{Factorization in \bm{$SL^\infty$}}

\author[R.~Lechner]{Richard Lechner}
\address{Richard Lechner,
  Institute of Analysis,
  Johannes Kepler University Linz,
  Altenberger Strasse 69,
  A-4040 Linz, Austria}
\email{Richard.Lechner@jku.at}

\date{\today}
\thanks{Supported by the Austrian Science Foundation (FWF) Pr.Nr. P28352}

\begin{abstract}
  We show that the non-separable Banach space $SL^\infty$ is primary.
  This is achieved by directly solving the infinite dimensional factorization problem in
  $SL^\infty$.
  In particular, we bypass Bourgain's localization method.
\end{abstract}

\maketitle

\section{Introduction}\label{sec:intro}

\noindent
Let $\dint$ denote the collection of dyadic subintervals of the unit interval~$[0,1)$, and let~$h_I$
denote the $L^\infty$-normalized Haar function supported on $I\in\dint$;
that is, $h_I$ is $+1$ on the left half of $I$, $h_I$ is $-1$ on the right half of $I$, and zero
otherwise.

The \emph{non-separable} Banach space $SL^\infty$ is the linear space
\begin{equation}\label{eq:sl-infty:space}
  \{f = \sum_{I\in\dint} a_I h_I : \|f\|_{SL^\infty} < \infty\},
\end{equation}
equipped with the norm
\begin{equation}\label{eq:sl-infty:norm}
  \big\| \sum_{I\in\dint} a_I h_I \big\|_{SL^\infty}
  = \| \big(\sum_{I\in\dint} a_I^2 h_I^2\big)^{1/2} \|_{L^\infty}.
\end{equation}
We want to emphasize that throughout this paper, whenever we encounter infinite sums in the Banach
space $SL^\infty$, we treat these series as a formal series representing the vector of coefficients,
and we do \emph{not} imply any kind of convergence.
The \emph{Hardy space} $H^1$ is the completion of
\begin{equation*}
  \spn\{ h_I : I \in \dint \}
\end{equation*}
under the norm
\begin{equation}\label{eq:Hp-norm}
  \|f\|_{H^1}
  = \int_0^1 \big(
    \sum_{I\in \dint} |a_I|^2 h_I^2(x)
  \big)^{1/2}
  \dif x
  ,
\end{equation}
where $f = \sum_{I\in \dint} a_I h_I$.
We note the well-known and obvious inequality (see e.g.~\cite{garsia:1973}):
\begin{equation}\label{eq:bracket-estimate}
  |\langle f, g\rangle| \leq \|f\|_{SL^\infty} \|g\|_{H^1},
  \qquad f\in SL^\infty,\ g\in H^1.
\end{equation}
Let $T$ denote a bounded, linear operator on $SL^\infty$.
We say an operator $T$ has \emph{large diagonal} with respect to the Haar system
$(h_I : I\in \dint)$ if there exists a $\delta > 0$ such that
\begin{equation*}
  |\langle T h_I, h_I \rangle|
  \geq \delta |I|,
  \qquad I\in \dint.
\end{equation*}

\section{Main Results}\label{sec:results}

\noindent
The first result Theorem~\ref{thm:factor} asserts that the identity operator on $SL^\infty$ factors
through any operator on $SL^\infty$ that has large diagonal with respect to the Haar system.

\begin{thm}\label{thm:factor}
  Let $\delta, \eta > 0$, and let $T : SL^\infty\to SL^\infty$ be an operator satisfying
  \begin{equation*}
    |\langle T h_I, h_I \rangle| \geq \delta |I|,
    \qquad I\in \dint.
  \end{equation*}
  Then the identity operator $\Id$ on $SL^\infty$ factors through~$T$, that is, there are
  operators $R, S : SL^\infty\to SL^\infty$ such that the diagram
  \begin{equation}\label{eq:factor}
    \vcxymatrix{SL^\infty \ar[r]^\Id \ar[d]_R & SL^\infty\\
      SL^\infty \ar[r]_T & SL^\infty \ar[u]_S}
  \end{equation}
  is commutative.
  Moreover, the operators $R$ and $S$ can be chosen with $\|R\|\|S\| \leq (1+\eta)/\delta$.
\end{thm}

Let us now recall the notion of a primary Banach space, see e.g.~\cite{lindenstrauss-tzafriri:1977}:
A Banach space $X$ is \emph{primary} if for every bounded projection $Q : X\to X$, either $Q(X)$ or
$(\Id - Q)(X)$ is isomorphic to $X$.

The subsequent factorization result Theorem~\ref{thm:primary} follows from Theorem~\ref{thm:factor} by means of
well established combinatorics of dyadic intervals, see e.g.~\cite{mueller:2005}.
Note that Theorem~\ref{thm:factor} is a theorem about an operator and a basis, whereas
Theorem~\ref{thm:primary} expresses an isomorphic invariant.
\begin{thm}\label{thm:primary}
  Let $T : SL^\infty\to SL^\infty$ be a bounded linear operator and $\eta > 0$.
  Then the identity operator $\Id$ on $SL^\infty$ factors through~$H=T$ or $H=\Id-T$, i.e., there
  exist operators $R, S : SL^\infty\to SL^\infty$ such that the diagram
  \begin{equation}\label{eq:primary}
    \vcxymatrix{SL^\infty \ar[r]^\Id \ar[d]_R & SL^\infty\\
      SL^\infty \ar[r]_H & SL^\infty \ar[u]_S}
  \end{equation}
  is commutative.
  Moreover, the operators $R$ and $S$ can be chosen with $\|R\| \|S\| \leq 2+\eta$.
  Consequently, the Banach space $SL^\infty$ is primary.
\end{thm}

Historically, the method used to prove factorization theorems or the primarity of separable Banach spaces
(e.g.~\cite{casazza:lin:1974,maurey:sous:1975,casazza:lin:1977,alspach:enflo:odell:1977,enflo:starbird:1979,jmst:1979,capon:1980:1,capon:1980:2,capon:1982,capon:1983,mueller:1994,laustsen:lechner:mueller:2015})
has been based on infinite dimensional reasoning, whereas the method used in non-separable Banach spaces
(e.g.~\cite{bourgain:1983,mueller:1988,blower:1990,arias:farmer:1996,wark:2007:class,wark:2007:direct-sum,mueller:2012,lechner:mueller:2014,lechner:2016-factor})
was finite dimensional in nature.
The localization method used for non-separable spaces goes back to Bourgain~\cite{bourgain:1983}.
There is but one exception where the reasoning in a non-separable Banach space is infinite
dimensional: Lindenstrauss~\cite{lindenstrauss:1967} showing that $\ell^\infty$ is prime.
Recall that an infinite dimensional Banach space $X$ is \emph{prime} if every infinite dimensional
complemented subspace is isomorphic to $X$, see e.g.~\cite{lindenstrauss-tzafriri:1977}.

The key point of this paper is that using Bourgain's localization method in non-separable Banach
spaces is not a naturally occurring necessity.
Specifically, we prove~Theorem~\ref{thm:factor} and Theorem~\ref{thm:primary} using just infinite dimensional methods.

\section{Block bases and projections in~$SL^\infty$}\label{sec:projections}

Here, we specify the conditions~\textref[J]{enu:j1}--\textref[J]{enu:j4} (which go back
to Jones~\cite{jones:1985}) under which a block basis of the Haar system in $SL^\infty$ spans a complemented copy of $SL^\infty$.
We also show that the conditions~\textref[J]{enu:j1}--\textref[J]{enu:j4} are stable under
reiteration.

\subsection{Jones' compatibility conditions for \bm{$SL^\infty$}}\label{subsec:compat}\hfill\\
\noindent
Let $\mathscr I\subset \dint$ be a collection of dyadic index intervals and let $\mathscr N$ be a
collection of sets.
For all $I\in\mathscr I$ let $\mathscr B_I\subset \mathscr N$.
We define
\begin{equation}\label{eq:abbreviations}
  \mathscr B
  = \bigcup_{I\in \mathscr I} \mathscr B_I
  \qquad\text{and}\qquad
  \qquad B_I = \bigcup_{N\in \mathscr B_I} N,
  \quad\text{for all $I\in \mathscr I$}.
\end{equation}
We say that the sequence $(\mathscr B_I : I \in \mathscr I)$ satisfies Jones' compatibility
conditions (see~\cite{jones:1985}) with constant $\kappa_J$, if the following conditions~\textref[J]{enu:j1}--\textref[J]{enu:j4} are satisfied.
\begin{enumerate}[\quad(J1)]
\item\label{enu:j1}
  The collection $\mathscr B$ consists of finitely many measurable and nested sets of positive
  measure.

\item\label{enu:j2}
  For each $I\in\mathscr I$, the collection $\mathscr B_I$ is non-empty and consists of pairwise
  disjoint sets.
  Furthermore, $\mathscr B_{I_0} \cap \mathscr B_{I_1} = \emptyset$, whenever $I_0,I_1\in\mathscr I$
  are distinct.

\item\label{enu:j3}
  For all $I_0, I_1\in \mathscr I$ holds that
  \begin{equation*}
    B_{I_0}\cap B_{I_1} = \emptyset\ \text{if}\ I_0 \cap I_1 = \emptyset,
    \qquad\text{and}\qquad
    B_{I_0} \subset B_{I_1}\ \text{if}\ I_0 \subset I_1.
  \end{equation*}

\item\label{enu:j4}
  For all $I_0,I\in \dint$ with $I_0\subset I$ and $N\in \mathscr B_I$, we have
  \begin{equation*}
    \frac{|N\cap B_{I_0}|}{|N|} \geq  \kappa_J^{-1} \frac{|B_{I_0}|}{|B_I|}.
  \end{equation*}
\end{enumerate}
For a discussion on the conditions~\textref[J]{enu:j1}--\textref[J]{enu:j4} and Jones'
conditions~\cite{jones:1985} in $\bmo$ see Remark~\ref{rem:jones}.

In the following Lemma~\ref{lem:projection-simple}, we record three facts about collections
satisfying~\textref[J]{enu:j1}--\textref[J]{enu:j4}.
\begin{lem}\label{lem:projection-simple}
  Let $(\mathscr B_I : I\in \dint)$ satisfy~\textref[J]{enu:j1}--\textref[J]{enu:j4}.
  Then the following statements are true:
  \begin{enumerate}[(i)]
  \item $(B_I : I\in \dint)$ is a sequence of nested measurable sets of positive measure.
    \label{enu:projection-simple:1}

  \item Let $I, I_0\in \dint$, then
    \begin{equation*}
      B_{I_0} \subset B_I
      \quad\text{if and only if}\quad
      I_0 \subset I.
    \end{equation*}
    \label{enu:projection-simple:2}

  \item Let $I_0, I\in \dint$, with $I_0\subset I$.
    Then for all $N_0\in \mathscr B_{I_0}$ there exists a set $N\in \mathscr B_I$ such that
    $N_0\subset N$.
    \label{enu:projection-simple:3}
  \end{enumerate}
\end{lem}

\begin{proof}
  Since $B_I$ is a finite union of measurable sets having positive measure, we only need to show the
  nestedness.
  Let $I_0,I_1\in \dint$ be such that $B_{I_0}\cap B_{I_1} \neq \emptyset$.
  If we assumed that the intersection $I_0\cap I_1$ were empty, then by~\textref[J]{enu:j3}
  we arrive at the contradiction $B_{I_0}\cap B_{I_1} = \emptyset$.
  Hence, we now know that $I_0\cap I_1\neq \emptyset$, which certainly implies that
  $I_0 \subset I_1$ or $I_1 \subset I_0$.
  Using~\textref[J]{enu:j3} concludes the proof of~\eqref{enu:projection-simple:1}.

  One of the implication of~\eqref{enu:projection-simple:2} follows from~\textref[J]{enu:j3}.
  We will now show the other one.
  To this end let $I, I_0\in \dint$ and $B_{I_0} \subset B_I$ and assume that $I_0 \not\subset I$.
  If $I_0\cap I = \emptyset$, then $B_{I_0} = B_{I_0}\cap B_I = \emptyset$ by~\textref[J]{enu:j3}, which
  contradicts~\textref[J]{enu:j1}.
  Thus we know $I_0 \supsetneq I$, and so we can find a $J\in \dint$ with $J\cap I = \emptyset$ and
  $I\cup J \subset I_0$.
  Hence, \textref[J]{enu:j3} yields $B_I\cup B_J\subset B_{I_0}$, which combined with our hypothesis
  $B_{I_0}\subset B_{I}$ gives us $B_J = \emptyset$, which contradicts~\textref[J]{enu:j1}.

  Finally, we will show~\eqref{enu:projection-simple:3}.
  Suppose that~\eqref{enu:projection-simple:3} is false.
  Then there exists $N_0\in \mathscr B_{I_0}$ such that
  \begin{equation}\label{proof:lem:projection-simple:1}
    N_0\not\subset N,
    \qquad N\in \mathscr B_I.
  \end{equation}
  By~\textref[J]{enu:j3} we have $B_{I_0}\subset B_I$, thus we know that there exists an
  $N\in \mathscr B_I$ such that $N_0\cap N\neq \emptyset$.
  Therefore, we obtain from~\eqref{proof:lem:projection-simple:1} and the nestedness of the
  collection $\mathscr N$ that $N\subset N_0$.
  But then~\textref[J]{enu:j3} gives us
  \begin{equation}\label{proof:lem:projection-simple:2}
    N\cap B_{I_1} \subset B_{I_0}\cap B_{I_1} = \emptyset.
  \end{equation}
  By~\textref[J]{enu:j4} and~\textref[J]{enu:j1} we obtain that
  \begin{equation*}
    \frac{|N\cap B_{I_1}|}{|N|} \geq \kappa_J^{-1} \frac{|B_{I_0}|}{|B_I|} > 0.
  \end{equation*}
  The latter inequality contradicts~\eqref{proof:lem:projection-simple:2}.
\end{proof}

\subsection{Reiterating Jones' compatibility conditions}\label{subsec:reiter}\hfill\\
Jones' compatibility conditions~\textref[J]{enu:j1}--\textref[J]{enu:j4} are stable under iteration
in the following sense.
\begin{thm}\label{thm:projection-iteration}
  Let $(\mathscr A_I : I \in \dint)$ be a sequence of collections of sets that
  satisfies~\textref[J]{enu:j1}--\textref[J]{enu:j4} with constant $\kappa_J$.
  Put $\mathscr M = \Union_{I\in \dint}\mathscr A_I$ and $A_I = \bigcup_{M\in \mathscr A_I} M$.
  Let $\mathscr N$ denote the collection of nested sets given by
  \begin{equation*}
    \mathscr N = \{A_I : I\in \dint\}.
  \end{equation*}
  For each $J\in \dint$ let $\mathscr B_J \subset \mathscr N$ be such that
  $(\mathscr B_J : J \in \dint)$ satisfies~\textref[J]{enu:j1}--\textref[J]{enu:j4} with constant
  $\kappa_J$, where we put $B_J = \bigcup_{A_I\in \mathscr B_J} A_I$.
  Finally, for all $J\in \dint$, we define
  \begin{equation*}
    \mathscr C_J = \bigcup_{A_I\in \mathscr B_J} \mathscr A_I
    \quad\text{and}\quad
    C_J = \bigcup_{A_I\in \mathscr B_J} A_I = B_J.
  \end{equation*}
  Then $(\mathscr C_J : J\in \dint)$ is a sequence of collections of sets in $\mathscr M$
  satisfying~\textref[J]{enu:j1}--\textref[J]{enu:j4} with constant $\kappa_J^2$.
\end{thm}

\begin{proof}
  By~\textref[J]{enu:j3} and Lemma~\ref{lem:projection-simple}~\eqref{enu:projection-simple:2} for
  $(\mathscr A_I)$ we obtain that $\mathscr N$ consists indeed of nested sets.
  Since $C_J = B_J$, it is clear that $(\mathscr C_J)$ satisfies~\textref[J]{enu:j3}
  and~\textref[J]{enu:j1}.

  We will now show that $(C_J)$ satisfies~\textref[J]{enu:j2}.
  To this end, let $M_0,M_1\in \mathscr C_J$ and assume that $M_0\cap M_1\neq \emptyset$.
  Per definition of $\mathscr C_J$, there exist $I_0,I_1\in \dint$ such that
  $A_{I_0},A_{I_1}\in \mathscr B_J$ and $M_i\in \mathscr A_i$, $i=0,1$.
  This implies $A_{I_0}\cap A_{I_1}\neq \emptyset$, so by the first part of~\textref[J]{enu:j2} for
  $(\mathscr B_J)$ we obtain $I_0=I_1$.
  Hence, $M_0,M_1\in \mathscr A_{I_0}$, and the second part of~\textref[J]{enu:j2} for
  $(\mathscr A_I)$ yields $M_0=M_1$.

  Next, we verify that $(\mathscr C_J)$ satisfies~\textref[J]{enu:j4}.
  Let $J_0,J\in \dint$ with $J_0\subset J$ and let $M\in \mathscr C_J$.
  We need to show that
  \begin{equation*}
    \frac{|M\cap C_{J_0}|}{|M|} \geq \kappa_J^{-2} \frac{|C_{J_0}|}{|C_J|}.
  \end{equation*}
  Per definition of $\mathscr C_J$, there exists a dyadic interval $I$ so that
  $A_I\in \mathscr B_J$ and $M\in \mathscr A_I$.
  Property \textref[J]{enu:j4} for the collection $(\mathscr B_J)$ and the definition of
  $B_{J_0}$ give
  \begin{equation*}
    \frac{|C_{J_0}|}{|C_J|}
    = \frac{|B_{J_0}|}{|B_J|}
    \leq \kappa_J \frac{|A_I\cap B_{J_0}|}{|A_I|}
    = \kappa_J \sum_{A_{I_0}\in \mathscr B_{J_0}} \frac{|A_I\cap A_{I_0}|}{|A_I|}.
  \end{equation*}
  Whenever $A_I\cap A_{I_0}\neq \emptyset$,
  Lemma~\ref{lem:projection-simple}~\eqref{enu:projection-simple:3} applied to $(\mathscr B_J)$ yields
  that $A_{I_0}\subset A_I$.
  By Lemma~\ref{lem:projection-simple}~\eqref{enu:projection-simple:2} applied to $(\mathscr A_J)$, we
  obtain that $A_{I_0}\subset A_I$ is equivalent to $I_0\subset I$.
  Thus we note
  \begin{equation}\label{proof:thm:projection-iteration:1}
    \frac{|C_{J_0}|}{|C_J|}
    \leq \kappa_J \sum_{\substack{I_0\subset I\\A_{I_0}\in \mathscr B_{J_0}}} \frac{|A_{I_0}|}{|A_I|}.
  \end{equation}
  Condition~\textref[J]{enu:j4} for the collection $(\mathscr A_I)$ and the definition of $C_{J_0}$
  give
  \begin{equation}\label{proof:thm:projection-iteration:2}
    \sum_{\substack{I_0\subset I\\A_{I_0}\in \mathscr B_{J_0}}} \frac{|A_{I_0}|}{|A_I|}
    \leq \kappa_J \sum_{\substack{I_0\subset I\\A_{I_0}\in \mathscr B_{J_0}}} \frac{|M\cap A_{I_0}|}{|M|}
    \leq \kappa_J \frac{|M\cap C_{J_0}|}{|M|}.
  \end{equation}
  Combining~\eqref{proof:thm:projection-iteration:1} and~\eqref{proof:thm:projection-iteration:2}
  concludes the proof.
\end{proof}

\subsection{Embeddings and projections in \bm{$SL^\infty$}}\label{subsec:projections}\hfill\\
\noindent
Here we establish that if $(\mathscr B_I : I\in \dint)$ satisfies Jones' compatibility
conditions~\textref[J]{enu:j1}--\textref[J]{enu:j4}, then the block basis
$(b_I : I\in \dint)$ of the Haar system $(h_I : I\in \dint)$ be given by
\begin{equation*}
  b_I = \sum_{K\in \mathscr B_I} h_K,
  \qquad I\in \dint,
\end{equation*}
spans a complemented copy of $SL^\infty$.

\begin{thm}\label{thm:projection}
  Let $\mathscr I\subset \dint$ be a collection of index intervals, and let
  $\mathscr B_I\subset \dint$, $I\in \mathscr I$.
  Assume that the sequence of collections of dyadic intervals $(\mathscr B_I : I\in \mathscr I )$
  satisfies Jones' compatibility conditions~\textref[J]{enu:j1}--\textref[J]{enu:j4} with constant
  $\kappa_J > 0$.
  Let the block basis $(b_I : I\in \mathscr I)$ of the Haar system $(h_I : I\in \dint)$ be given by
  \begin{equation}\label{eq:thm:projection:block-basis}
    b_I = \sum_{K\in \mathscr B_I} h_K,
    \qquad I\in \mathscr I.
  \end{equation}
  Then the operators $B,Q : SL^\infty\to SL^\infty$ given by
  \begin{equation}\label{eq:thm:projection:operators}
    Bf = \sum_{I\in \mathscr I} \frac{\langle f, h_I\rangle}{\|h_I\|_2^2} b_I
    \qquad\text{and}\qquad
    Qf = \sum_{I\in \mathscr I} \frac{\langle f, b_I\rangle}{\|b_I\|_2^2} h_I
  \end{equation}
  satisfy the estimates
  \begin{equation}\label{eq:thm:projection:estimates}
    \|B f \|_{SL^\infty} \leq \|f\|_{SL^\infty}
    \quad\text{and}\quad
    \|Q f \|_{SL^\infty} \leq \kappa_J^{1/2} \|f\|_{SL^\infty},
  \end{equation}
  for all $f\in SL^\infty$.
  Moreover, the diagram
  \begin{equation}\label{eq:thm:projection:diagram}
    \vcxymatrix{SL^\infty \ar[rr]^{\Id_{SL^\infty}} \ar[rd]_{B} & & SL^\infty\\
      &  SL^\infty  \ar[ru]_{Q} &
    }
  \end{equation}
  is commutative.
  Consequently, the range of $B$ is complemented, and $B$ is an isomorphism onto its range.
\end{thm}

\begin{proof}
  First, we will show the estimate for $B$.
  To this end, let $f\in SL^\infty$ be finitely supported with respect to the Haar system
  $(h_I : I\in \dint)$, i.e.
  \begin{equation*}
    f = \sum_{I\in \dint^{N_0}} a_I h_I,
  \end{equation*}
  for some integer $N_0$ and scalars $(a_I : I\in \dint^{N_0})$.
  By~\textref[J]{enu:j2}, we have
  \begin{equation}\label{proof:thm:projection:1}
    \|B f\|_{SL^\infty}^2
    = \sup_{x\in [0,1)} \sum_{I\in \mathscr I\cap \dint^{N_0}} a_I^2 \charfun_{B_I}(x).
  \end{equation}
  Given $x\in[0,1)$, we define $I(x)\in \dint^{N_0}\cup\{\emptyset\}$ by
  \begin{equation*}
    I(x) = \bigcap \{ J : J\in \mathscr I\cap\dint^{N_0},\, B_J\ni x\}.
  \end{equation*}
  By definition of $I(x)$, we have that $I\supset I(x)$, whenever $B_I\ni x$.
  Thus, the following inequalities hold:
  \begin{equation*}
    \sum_{I\in \mathscr I\cap \dint^{N_0}} a_I^2 \charfun_{B_I}(x)
    \leq \sum_{\substack{I\in \mathscr I\\I\supset I(x)}} a_I^2 \charfun_I(y)
    \leq \sum_{I\in \mathscr I} a_I^2 \charfun_I(y),
    \qquad y\in I(x).
  \end{equation*}
  Taking the supremum over all $x\in [0,1)$ in the latter estimate yields in combination
  with~\eqref{proof:thm:projection:1} that
  \begin{equation}\label{proof:thm:projection:2}
    \|B f\|_{SL^\infty}^2
    \leq \sup_{x\in [0,1)} \inf_{y\in I(x)} \sum_{I\in \mathscr I} a_I^2 \charfun_I(y)
    \leq \|f\|_{SL^\infty}^2,
  \end{equation}
  for all finitely supported $f\in SL^\infty$.
  To show the above estimate for arbitrary $f\in SL^\infty$, consider the following.
  Let $f\in SL^\infty$ and define $f^{N_0}$, by
  \begin{equation*}
    f^{N_0} = \sum_{I\in \dint^{N_0}} \frac{\langle f, h_I\rangle}{\|h_I\|_2^2} h_I,
    \qquad N_0\in \mathbb N.
  \end{equation*}
  Observe that by definition of the norm $\|\cdot\|_{SL^\infty}$ and
  by~\eqref{proof:thm:projection:2} we obtain
  \begin{equation*}
    \|B f\|_{SL^\infty}
    = \sup_{N_0\in \mathbb N} \|B f^{N_0}\|_{SL^\infty}
    \leq \sup_{N_0\in \mathbb N} \|f^{N_0}\|_{SL^\infty}
    = \|f\|_{SL^\infty}.
  \end{equation*}

  Before we continue with the estimate for $Q$, we will now introduce some notation.
  Given $N_0\in \mathbb N$ let $\mathscr B_{N_0}$ denote the collection of building blocks given by
  \begin{equation*}
    \mathscr B_{N_0}
    = \{ K_0\in \mathscr B_{I_0} : I_0\in \mathscr I\cap \dint_{N_0}\}.
  \end{equation*}
  Accordingly, we define the building blocks $\mathscr B^{N_0}$ by
  \begin{equation*}
    \mathscr B^{N_0}
    = \{ K\in \mathscr B_I : I\in \mathscr I\cap \dint^{N_0}\}.
  \end{equation*}
  We will now estimate $Q$.

  To begin with, let us assume that $f$ is of the following form:
  \begin{equation*}
    f = \sum_{K\in \mathscr B^{N_0}} a_K h_K.
  \end{equation*}
  On the one hand, a straightforward calculation using only the properties of dyadic intervals shows
  \begin{align*}
    \|Qf\|_{SL^\infty}^2
    & = \sup_{x\in [0,1)} \sum_{I\in \mathscr I\cap\dint^{N_0}}
      \frac{\langle f, b_I \rangle^2}{\|b_I\|^2} \charfun_I(x)\\
    & = \sup_{x\in [0,1)} \sum_{\substack{I\in \mathscr I\cap\dint^{N_0}\\I\ni x}} \Big(
    \sum_{K\in \mathscr B_I} a_K \frac{|K|}{|B_I|}
    \Big)^2\\
    & = \max_{I_0\in \mathscr I\cap \dint_{N_0}} \sum_{\substack{I\in \mathscr I\\I\supset I_0}} \Big(
    \sum_{K\in \mathscr B_I} a_K \frac{|K|}{|B_I|}
    \Big)^2.
  \end{align*}
  Applying the Cauchy-Schwarz inequality to the inner sum yields
  \begin{equation}\label{proof:thm:projection:3}
    \|Qf\|_{SL^\infty}^2
    \leq \max_{I_0\in \mathscr I\cap \dint_{N_0}} \sum_{\substack{I\in \mathscr I\\I\supset I_0}}
    \sum_{K\in \mathscr B_I} a_K^2 \frac{|K|}{|B_I|}.
  \end{equation}
  On the other hand, by the definition of $\|\cdot\|_{SL^\infty}$ and by~\textref[J]{enu:j2} we obtain
  \begin{equation*}
    \|f\|_{SL^\infty}^2
    = \sup_{x\in[0,1)} \sum_{K\in \mathscr B^{N_0}} a_K^2 \charfun_K(x)
    = \sup_{x\in[0,1)} \sum_{I\in \mathscr I\cap \dint^{N_0}} \sum_{K\in \mathscr B_I} a_K^2 \charfun_K(x).
  \end{equation*}
  By~\textref[J]{enu:j3} and~\textref[J]{enu:j1}, the collections
  $B_{I_0}$, $I_0\in \mathscr I\cap \dint_{N_0}$ are pairwise disjoint and of positive measure,
  hence we obtain for all measurable, non-negative functions $g$ that
  \begin{align*}
    \sup_{x\in[0,1)} g(x)
    \geq \sup_{x\in[0,1)} \sum_{I_0\in \mathscr I\cap \dint_{N_0}} \Big(
      \frac{1}{|B_{I_0}|}\int_{B_{I_0}} g(y) \dif y
    \Big)
    \charfun_{B_{I_0}}(x).
  \end{align*}
  Using the latter estimate for
  $g = \sum_{I\in \mathscr I\cap \dint^{N_0}} \sum_{K\in \mathscr B_I} a_K^2 \charfun_K(x)$ yields
  \begin{equation*}
    \|f\|_{SL^\infty}^2
    \geq \sup_{x\in[0,1)} \sum_{I_0\in \mathscr I\cap \dint_{N_0}}
    \sum_{I\in \mathscr I\cap \dint^{N_0}} \sum_{K\in \mathscr B_I}
      a_K^2 \frac{|K\cap B_{I_0}|}{|B_{I_0}|} \charfun_{B_{I_0}}(x).
  \end{equation*}
  By~\textref[J]{enu:j3} and Lemma~\ref{lem:projection-simple}~\eqref{enu:projection-simple:2} we obtain
  for all $K\in \mathscr B_I$ that whenever $K\cap B_{I_0}\neq \emptyset$, we have that
  $I\supset I_0$ and $B_I\supset B_{I_0}$.
  Therefore and by~\textref[J]{enu:j3}, we get
  \begin{align*}
    \|f\|_{SL^\infty}^2
    & \geq \sup_{x\in[0,1)} \sum_{I_0\in \mathscr I\cap \dint_{N_0}}
    \sum_{\substack{I\in \mathscr I\\I\supset I_0}} \sum_{K\in \mathscr B_I}
    a_K^2 \frac{|K\cap B_{I_0}|}{|B_{I_0}|} \charfun_{B_{I_0}}(x)\\
    & = \max_{I_0\in \mathscr I\cap \dint_{N_0}}
    \sum_{\substack{I\in \mathscr I\\I\supset I_0}} \sum_{K\in \mathscr B_I}
    a_K^2 \frac{|K\cap B_{I_0}|}{|B_{I_0}|}.
  \end{align*}
  Finally, using~\textref[J]{enu:j4} yields the following lower estimate:
  \begin{equation*}
    \|f\|_{SL^\infty}^2
    \geq \kappa_J^{-1} \max_{I_0\in \mathscr I\cap \dint_{N_0}}
    \sum_{\substack{I\in \mathscr I\\I\supset I_0}} \sum_{K\in \mathscr B_I}
    a_K^2 \frac{|K|}{|B_I|}.
  \end{equation*}
  Comparing the above estimate with~\eqref{proof:thm:projection:3} shows
  \begin{equation}\label{proof:thm:projection:4}
    \|Qf\|_{SL^\infty}\leq \kappa_J^{1/2} \|f\|_{SL^\infty},
  \end{equation}
  for all finitely supported $f\in SL^\infty$.
  Now let $\mathscr B = \bigcup_{N_0\in \mathbb N} \mathscr B^{N_0}$,
  \begin{equation*}
    f = \sum_{K\in \mathscr B} a_K h_K
    \quad\text{and}\quad
    f^{N_0} = \sum_{K\in \mathscr B^{N_0}} a_K h_K.
  \end{equation*}
  Note the identity
  \begin{equation*}
    Q f^{N_0}
    = \sum_{I\in \mathscr I\cap \dint^{N_0}}
    \frac{\langle f, b_I\rangle}{\|b_I\|_2^2} h_I.
  \end{equation*}
  The above identity, the definition of $\|\cdot\|_{SL^\infty}$ and
  inequality~\eqref{proof:thm:projection:4} yield
  \begin{align*}
    \|Q f\|_{SL^\infty}
    = \sup_{N_0\in \mathbb N} \|Q f^{N_0}\|_{SL^\infty}
    \leq \kappa_J^{1/2} \sup_{N_0\in \mathbb N} \|f^{N_0}\|_{SL^\infty}
    \leq \kappa_J^{1/2} \|f\|_{SL^\infty}.
  \end{align*}
\end{proof}

\begin{rem}\label{rem:jones}
  The conditions~\textref[J]{enu:j1}--\textref[J]{enu:j4} go back to Jones~\cite{jones:1985}.
  In~\cite{jones:1985}, Jones treated projections in $\bmo$ with the following three conditions:
  
  Suppose that $\mathscr C_1,\ldots,\mathscr C_N$ are disjoint collections of intervals which
  satisfy
  \begin{enumerate}[(\cite{jones:1985}, 2.1)]
  \item $\|\sum_{I\in \mathscr I_j} \charfun_I\|_{L^\infty} = 1$, $1\leq j\leq N$,
    \label{enu:jones:1}

  \item and suppose that there are constants $a_{jk}$ such that whenever $I\in \mathscr C_j$,
    \begin{equation*}
      \frac{1}{2} a_{jk}
      \leq \frac{1}{|I|} \sum_{\substack{J\subset I\\J\in \mathscr C_k}} |J|
      \leq a_{jk},
      \qquad 1\leq j,k \leq N.
    \end{equation*}
    \label{enu:jones:2}

  \item Furthermore we suppose whenever $2\leq j\leq N$ and $I\in \mathscr C_j$ there is
    $J\in \mathscr C_{j-1}$ such that $I\subset J$.
    \label{enu:jones:3}
  \end{enumerate}
  We remark that Jones' conditions~\textref[\cite{jones:1985}, 2.]{enu:jones:1},
  \textref[\cite{jones:1985}, 2.]{enu:jones:2}, and~\textref[\cite{jones:1985}, 2.]{enu:jones:3}
  imply~\textref[J]{enu:j1}--\textref[J]{enu:j4} (with a reasonable interpretation of the binary
  tree structure).
  Most noteworthy are the following observations:
  \begin{itemize}
  \item Condition~\textref[\cite{jones:1985}, 2.]{enu:jones:1} together with the disjointness of the
    collections $\mathscr C_j$ (the line above~\textref[\cite{jones:1985}, 2.]{enu:jones:1}) is
    exactly~\textref[J]{enu:j2}.

  \item The absence of the corresponding upper estimate
    of~\textref[\cite{jones:1985}, 2.]{enu:jones:2} in~\textref[J]{enu:j4}.

  \item The ``uniform packing'' condition~\textref[\cite{jones:1985}, 2.]{enu:jones:2} together with
    the ``stacking'' condition~\textref[\cite{jones:1985}, 2.]{enu:jones:3} imply
    condition~\textref[J]{enu:j1} and condition~\textref[J]{enu:j4}.

  \item The absence of a corresponding ``stacking'' condition
    in~\textref[J]{enu:j1}--\textref[J]{enu:j4}.

  \item The conditions~\textref[J]{enu:j1}--\textref[J]{enu:j4} imply a partially ordered
    (with respect to inclusion) variant of Jones' ``stacking'' condition, if $\mathscr I = \dint$.
    However, that is \emph{not} the case if $\mathscr I$ is linearly ordered with respect to
    inclusion.
    We refer the reader to the proof of
    Lemma~\ref{lem:projection-simple}~\eqref{enu:projection-simple:3}).
  \end{itemize}
\end{rem}

\begin{rem}\label{rem:projection}
  Let $(\mathscr B_I : I \in \dint)$ denote a sequence of collections of dyadic intervals satisfying
  Jones' compatibility conditions~\textref[J]{enu:j1}--\textref[J]{enu:j4}.
  Given a sequence of signs
  $\varepsilon = (\varepsilon_K : \varepsilon_K\in \{\pm 1\},\ K\in \dint)$ we define
  \begin{equation}\label{eq:block-basis}
    b_I^{(\varepsilon)} = \sum_{K\in \mathscr B_I} \varepsilon_K h_K.
  \end{equation}
  We call $(b_I^{(\varepsilon)})$ the \emph{block basis generated by} $(\mathscr B_I : I\in \dint)$
  and~$\varepsilon$.
 
  The block basis $(b_I^{(\varepsilon)})$ gives rise to the operators
  $B^{(\varepsilon)}, Q^{(\varepsilon)} : SL^\infty\to SL^\infty$:
  \begin{equation}\label{rem:projection:operators:signs}
    B^{(\varepsilon)} f
    = \sum_{I\in \mathscr I} \frac{\langle f, h_I\rangle}{\|h_I\|_2^2} b_I^{(\varepsilon)}
    \qquad\text{and}\qquad
    Q^{(\varepsilon)} f
    = \sum_{I\in \mathscr I}
    \frac{\langle f, b_I^{(\varepsilon)}\rangle}{\|b_I^{(\varepsilon)}\|_2^2} h_I.
  \end{equation}
  See Theorem~\ref{thm:projection} for a definition of the operators $B,Q$.
  By the $1$-unconditionality of the Haar system in $SL^\infty$ and
  \begin{equation*}
    Q^{(\varepsilon)} f
    = Q f^{(\varepsilon)},
    \quad\text{where}\quad
    f^{(\varepsilon)} = \sum_{K\in \dint} \varepsilon_K \frac{\langle f, h_K\rangle}{\|h_K\|_2^2} h_K,
  \end{equation*}
  we obtain the estimates
  \begin{equation}\label{rem:projection:estimates:signs}
    \|B^{(\varepsilon)} f \|_{SL^\infty} \leq \|B f\|_{SL^\infty}
    \quad\text{and}\quad
    \|Q^{(\varepsilon)} f \|_{SL^\infty} \leq \|Q f\|_{SL^\infty},
  \end{equation}
  for all $f\in SL^\infty$.
  Moreover, the have the identity
  \begin{equation}\label{rem:projection:identity:signs}
    Q^{(\varepsilon)} B^{(\varepsilon)} = \Id_{SL^\infty}.
  \end{equation}
  Consequently, the range of $B^{(\varepsilon)}$ is complemented, and $B^{(\varepsilon)}$ is an
  isomorphism onto its range.
\end{rem}

\section{Factorization of the identity operator on $SL^\infty$ through operators with large diagonal}\label{sec:factor}

\noindent
Here, we will develop the crucial tools that permit us to almost-diagonalize a given operator $T$ on
$SL^\infty$, see Theorem~\ref{thm:quasi-diag}.
The almost-diagonalization result Theorem~\ref{thm:quasi-diag}, will then be used to show our first main
result Theorem~\ref{thm:factor}: we prove that the identity operator on $SL^\infty$ factors through
operators $T$ acting on $SL^\infty$ which have large diagonal with respect to the Haar system.
By well established methods, we therefore obtain that $SL^\infty$ is primary, which proves our
second main result Theorem~\ref{thm:primary}.

We emphasize that all of our proofs bypass Bourgain's localization method, which has been used many
times for showing the primarity of non-separable Banach spaces, see~e.g.~\cite{bourgain:1983,mueller:1988,blower:1990,arias:farmer:1996,wark:2007:class,wark:2007:direct-sum,mueller:2012,lechner:mueller:2014,lechner:2016-factor}.
Lemma~\ref{lem:sieve} is the key ingredient which allows us to use infinite dimensional reasoning in the
\emph{non-separable} space $SL^\infty$.
An $\ell^\infty$ variant of Lemma~\ref{lem:sieve} was used by Lindenstrauss to prove that $\ell^\infty$ is
prime~\cite{lindenstrauss:1967}.

\subsection{Almost-annihilating subspaces of \bm{$H^1$} and \bm{$SL^\infty$}}\label{subsec:preparation}\hfill\\
\noindent
Firstly, we prove that Rademacher functions $r_m$ converge to $0$, when tested against functions
$f\in SL^\infty$.
Secondly, we show how to select large subsets of the dyadic intervals, so that a given operator
$T : SL^\infty\to SL^\infty$ is small when acting on the subspace spanned by these intervals (and is tested against a function in $H^1$).

For any sequence of scalars $c=(c_I : I\in \dint)$, the Rademacher type function $r_m^{(c)}$ is
given by
\begin{equation}\label{eq:rademacher-type-function}
  r_m^{(c)} = \sum_{I\in \dint_m} c_I h_I,
  \qquad m\in \mathbb N.
\end{equation}

\begin{lem}\label{lem:bracket-convergence}
  Let $f\in SL^\infty$ and $g\in H^1$.
  Then
  \begin{equation}\label{eq:bracket-convergence}
    \sup_{\|c\|_{\ell^\infty}\leq 1} |\langle f, r_m^{(c)}\rangle| \to 0
    \quad\text{and}\quad
    \sup_{\|c\|_{\ell^\infty}\leq 1} |\langle T r_m^{(c)}, g \rangle| \to 0,
    \qquad\text{as $m\to \infty$}.
  \end{equation}
\end{lem}

\begin{proof}
  Let $f \in SL^\infty$ and $g\in H^1$.
  Note that there are sequences of scalars $\theta = (\theta_I : I\in \dint)$ and
  $\varepsilon = (\varepsilon_I : I\in \dint)$
  with $|\theta_I|=|\varepsilon_I|=1$, $I\in \dint$ such that
  \begin{equation}\label{proof:lem:bracket-convergence:1}
    \sup_{\|c\|_{\ell^\infty}\leq 1} |\langle f, r_m^{(c)}\rangle|
    = |\langle f, r_m^{(\theta)}\rangle|
    \qquad\text{and}\qquad
    \sup_{\|c\|_{\ell^\infty}\leq 1} |\langle T r_m^{(c)}, g\rangle|
    = |\langle T r_m^{(\varepsilon)}, g\rangle|.
  \end{equation}
  for all $m\in \mathbb N$.

  Now let $(\omega_m)_{m=1}^M$ denote a finite sequence of scalars and consider that
  by~\eqref{eq:bracket-estimate} we have
  \begin{equation*}
    \sum_{m=1}^M \omega_m \langle f, r_m^{(\theta)} \rangle
    \leq \|f\|_{SL^\infty} \big\| \sum_{m=1}^M \omega_m r_m^{(\theta)} \big\|_{H^1}
    \leq \|f\|_{SL^\infty} \big( \sum_{m=1}^M \omega_m^2 \big)^{1/2}.
  \end{equation*}
  Putting $\omega_m = \langle f, r_m^{(\theta)} \rangle$ gives
  \begin{equation*}
    \big( \sum_{m=1}^M |\langle f, r_m^{(\theta)} \rangle|^2 \big)^{1/2}
    \leq \|f\|_{SL^\infty}.
  \end{equation*}
  Combining the latter estimate with~\eqref{proof:lem:bracket-convergence:1} yields the first part
  of~\eqref{eq:bracket-convergence}.

  The argument for the second part is similar.
  By~\eqref{eq:bracket-estimate}, we obtain that
  \begin{equation*}
    \sum_{m=1}^M \omega_m \langle T r_m^{(\varepsilon)}, g \rangle
    \leq \|T\| \|g\|_{H^1} \big\| \sum_{m=1}^M \omega_m r_m^{(\varepsilon)} \big\|_{SL^\infty}
    \leq \|T\| \|g\|_{H^1} \big( \sum_{m=1}^M \omega_m^2 \big)^{1/2}.
  \end{equation*}
  By putting $\omega_m = \langle T r_m^{(\varepsilon)}, g \rangle$ we obtain
  \begin{equation*}
    \big( \sum_{m=1}^M |\langle T r_m^{(\varepsilon)}, g \rangle|^2 \big)^{1/2}
    \leq \|T\| \|g\|_{H^1},
  \end{equation*}
  which when combined with~\eqref{proof:lem:bracket-convergence:1} concludes the proof.
\end{proof}

Here we come to the crucial Lemma that enables infinite dimensional reasoning in the non-separable
Banach space $SL^\infty$.
\begin{lem}\label{lem:sieve}
  Let $\eta > 0$, $g\in H^1$ and let $\Gamma\subset \mathbb N$ be infinite.
  Suppose that $T : SL^\infty\to SL^\infty$ is a bounded linear operator.
  Then there exists an infinite set $\Lambda \subset \Gamma$ such that
  \begin{equation*}
    \sup_{\|f\|_{SL^\infty} \leq 1} |\langle T P_\Lambda f, g\rangle|
    \leq \eta \|g\|_{H^1},
  \end{equation*}
  where the norm one projection $P_\Lambda : SL^\infty\to SL^\infty$ is given by
  \begin{equation*}
    P_\Lambda (\sum_{I\in \dint} a_I h_I) = \sum_{m\in \Lambda} \sum_{I\in \dint_m} a_I h_I.
  \end{equation*}
\end{lem}

\begin{proof}[Proof of Lemma~\ref{lem:sieve}]
  Let $\eta > 0$ and $g\in H^1$ and assume that $\Gamma = \mathbb N$.
  Suppose the conclusion of the Lemma is false.
  Define $k = \lceil \frac{\|T\|^2}{\eta^2} \rceil$ and choose infinite, disjoint sets
  $\Lambda_1,\Lambda_2,\dots,\Lambda_k$.
  By our assumption, we can find $f_1,f_2,\dots,f_k \in SL^\infty$ with
  $\|f_j\|_{SL^\infty} = 1$, $1\leq j\leq k$ such that
  \begin{equation*}
    \langle T P_{\Lambda_j} f_j, g\rangle
    > \eta \|g\|_{H^1},
    \qquad\text{for all $1\leq j \leq k$}.
  \end{equation*}
  Summing these estimates and using~\eqref{eq:bracket-estimate} yields
  \begin{equation}\label{proof:lem:sieve:1}
    \|g\|_{H^1} \|T\| \big\|\sum_{j=1}^k P_{\Lambda_j} f_j\big\|_{SL^\infty}
    > k \eta \|g\|_{H^1}.
  \end{equation}
  Since the $\Lambda_j$, $1\leq j\leq k$ are disjoint, we have that
  \begin{equation*}
    \sqfun \Big( \sum_{j=1}^k P_{\Lambda_j} f_j \Big)
    = \Big( \sum_{j=1}^k \sqfun (P_{\Lambda_j} f_j)^2 \Big)^{1/2},
  \end{equation*}
  and therefore we obtain
  \begin{equation}\label{proof:lem:sieve:2}
    \big\|\sum_{j=1}^k P_{\Lambda_j} f_j\big\|_{SL^\infty}
    \leq \Big( \sum_{j=1}^k \| P_{\Lambda_j} f_j\|_{SL^\infty}^2 \Big)^{1/2}
    \leq k^{1/2}.
  \end{equation}
  By combining the estimates~\eqref{proof:lem:sieve:1} and~\eqref{proof:lem:sieve:2}, we reach a
  contradiction.
\end{proof}

\subsection{Diagonalization of operators on \bm{$SL^\infty$}}\label{subsec:diagonalization}\hfill\\
\noindent
We will show that any given operator $T$ acting on $SL^\infty$ with large diagonal can be
almost-diagonalized by a block basis of the Haar system $(b_I : I\in \dint)$, that spans a
complemented copy of $SL^\infty$ (see Theorem~\ref{thm:quasi-diag}).
This is achieved by constructing $(b_I : I\in \dint)$ with aid from the results in
Section~\ref{subsec:preparation}, so that Jones' compatibility conditions~\textref[J]{enu:j1}--\textref[J]{enu:j4} are satisfied.

From here on, we will regularly identify a dyadic interval $I \in \dint$ with its natural ordering
number $\diindex(I)$ given by
\begin{equation*}
  \diindex(I) = 2^n - 1 + k,
  \qquad\text{if $I=[k2^{-n}, (k+1)2^{-n}]$}.
\end{equation*}
To be precise, for $\diindex(I) = i$ we identify
\begin{equation*}
  \mathscr B_I = \mathscr B_i
  \qquad\text{and}\qquad
  b_I^{(\varepsilon)} = b_i^{(\varepsilon)}.
\end{equation*}
The block basis $(b_I^{(\varepsilon)} : I\in \dint)$ is defined in~\eqref{rem:projection}.
See also below.

\begin{thm}\label{thm:quasi-diag}
  Let $\delta \geq 0$ and let $T : SL^\infty\to SL^\infty$ be an operator satisfying
  \begin{equation*}
    |\langle T h_I, h_I \rangle| \geq \delta |I|,
    \qquad I\in \dint.
  \end{equation*}
  Then for any $\eta > 0$, there exists a sequence of collections $(\mathscr B_I : I\in \dint)$
  and a sequence of signs
  $\varepsilon = (\varepsilon_K : \varepsilon_K \in \{\pm 1\},\ K\in \dint)$ which generate the
  block basis of the Haar system $(b_I^{(\varepsilon)} : I\in \dint)$ given by
  \begin{equation*}
    b_I^{(\varepsilon)} = \sum_{K \in \mathscr B_I} \varepsilon_K h_K,
    \qquad I\in \dint,
  \end{equation*}
  so that the following conditions are satisfied:
  \begin{enumerate}[(i)]
  \item $(\mathscr B_I : I\in \dint)$ satisfies Jones' compatibility
    conditions~\textref[J]{enu:j1}--\textref[J]{enu:j4} with constant $\kappa_J = 1$.
    \label{enu:thm:quasi-diag-i}

  \item $(b_I^{(\varepsilon)} : I\in \dint)$ almost--diagonalizes $T$ so that $T$ has
    large diagonal with respect to $(b_I^{(\varepsilon)} : I\in \dint)$.
    To be more precise, for any $i\in \mathbb N_0$ we have the estimates
    \begin{subequations}\label{eq:thm:quasi-diag-ii}
      \begin{align}
        \sum_{j=0}^{i-1} |\langle T b_j^{(\varepsilon)}, b_i^{(\varepsilon)}\rangle|
        & \leq \eta 4^{-i} \|b_i^{(\varepsilon)}\|_2^2,
        \label{eq:thm:quasi-diag-ii:a}\\
        \langle T b_i^{(\varepsilon)}, b_i^{(\varepsilon)} \rangle
        & \geq \delta\|b_i^{(\varepsilon)}\|_2^2,
          \label{eq:thm:quasi-diag-ii:b}\\
        \sup \Big\{ |\langle T g , b_i^{(\varepsilon)}\rangle| :
          g = \sum_{j=i+1}^\infty a_j b_j^{(\varepsilon)},\ \|g\|_{SL^\infty}\leq 1
        \Big\}
        & \leq \eta 4^{-i} \|b_i^{(\varepsilon)}\|_2^2.
          \label{eq:thm:quasi-diag-ii:c}
      \end{align}
    \end{subequations}
    \label{enu:thm:quasi-diag-ii}
  \end{enumerate}
\end{thm}

\subsection*{Proof of Theorem~\ref{thm:quasi-diag}}\hfill\\
\noindent
Let $\delta \geq 0$, $\eta > 0$ and $T : SL^\infty\to SL^\infty$.
Before we begin with the actual proof, observe that by $1$-unconditionality, we can assume that
\begin{equation*}
\langle T h_I, h_I \rangle \geq \delta |I|,\qquad I\in \dint.
\end{equation*}
Given $I \in \dint$, we write
\begin{subequations}\label{eq:decomp}
  \begin{equation}
    T h_I = \alpha_I h_I + r_I,
  \end{equation}
  where
  \begin{equation}
    \alpha_I = \frac{\langle T h_I, h_I \rangle}{|I|}
    \quad\text{and}\quad
    r_I = \sum_{J\neq I}
    \frac{\langle T h_I, h_J \rangle}{|J|} h_J.
  \end{equation}
\end{subequations}
We note the estimate
\begin{equation}\label{eq:a-estimate}
  \delta \leq \alpha_I \leq \|T\|.
\end{equation}

\subsubsection*{Inductive construction of $(b_I^{(\varepsilon)} : I\in \dint)$}\hfill\\
\noindent
To begin the induction, we simply put
\begin{equation*}
  \mathscr B_0 = \mathscr B_{[0,1)} = \{[0,1)\}
  \qquad\text{and}\qquad
  b_0^{(\varepsilon)} = b_{[0,1)}^{(\varepsilon)} = h_{[0,1)}.
\end{equation*}
We complete the initial step of our construction, by choosing $\Lambda_1 \subset \mathbb N$
according to Lemma~\ref{lem:sieve} such that
\begin{equation*}
  \sup_{\|f\|_{SL^\infty}\leq 1} |\langle T P_{\Lambda_1} f, b_0^{(\varepsilon)}\rangle|
  \leq \eta \|b_0^{(\varepsilon)}\|_2^2.
\end{equation*}

For the inductive step, let us now assume that we have already
\begin{itemize}
\item chosen a strictly increasing sequence of integers $(m_j)$ and infinite index sets
  $\Lambda_1 \supset\dots\supset \Lambda_i$ with $m_j\in \Lambda_j\setminus \Lambda_{j+1}$,
  $1\leq j\leq i-1$,

\item constructed finite collections $\mathscr B_j$ with $\mathscr B_j \subset \dint_{m_j}$,
  $0\leq j \leq i-1$,

\item made a suitable choice of signs
  $\varepsilon = (\varepsilon_K : \varepsilon_K \in \{\pm 1\}, K\in \bigcup_{j < i}\mathscr B_j)$,

\item and the block basis elements $b_j^{(\varepsilon)}$ have the form
  \begin{equation*}
    b_j^{(\varepsilon)} = \sum_{K\in \mathscr B_j} \varepsilon_K h_K,
    \qquad 0\leq j\leq i-1.
  \end{equation*}
\end{itemize}

We will now choose an integer $m_i \in \Lambda_i$ with $m_i > m_{i-1}$, construct a finite
collection $\mathscr B_i \subset \dint_{m_i}$, choose signs
$\varepsilon = (\varepsilon_K : \varepsilon_K \in \{\pm 1\}, K\in \mathscr B_i)$,
select an infinite subset $\Lambda_{i+1}\subset \Lambda_i\setminus\{m_i\}$ and define
$b_i^{(\varepsilon)}$ by
\begin{subequations}\label{eq:induction-properties}
  \begin{align}
    b_i^{(\varepsilon)}
    & = \sum_{K\in \mathscr B_i} \varepsilon_K h_K,
      \label{eq:induction-properties:a}
  \end{align}
  such that the operator $T$ is almost-diagonalized by the block basis $(b_i^{(\varepsilon)})$
  while preserving the large diagonal.
  To be precise:
  \begin{align}
    \sum_{j=0}^{i-1} |\langle T b_j^{(\varepsilon)}, b_i^{(\varepsilon)}\rangle|
    & \leq \eta 4^{-i} \|b_i^{(\varepsilon)}\|_2^2,
      \label{eq:induction-properties:c}\\
    \langle T b_i^{(\varepsilon)}, b_i^{(\varepsilon)} \rangle
    & \geq \delta\|b_i^{(\varepsilon)}\|_2^2,
      \label{eq:induction-properties:d}\\
    \sup_{\|f\|_{SL^\infty}\leq 1} |\langle T P_{\Lambda_{i+1}} f, b_i^{(\varepsilon)}\rangle|
    & \leq \eta 4^{-i} \|b_i^{(\varepsilon)}\|_2^2.
      \label{eq:induction-properties:e}
  \end{align}
\end{subequations}
For a definition of the projection $P_\Lambda$ see Lemma~\ref{lem:sieve}.
For the most part of this inductive construction step, we will assume that
$\Lambda_i = \mathbb N$.

Now, let $I\in \dint$ be such that $\diindex(I) = i$.
The dyadic interval $\widetilde I$ denotes the unique dyadic interval such that
$\widetilde I \supset I$ and $|\widetilde I| = 2 |I|$.
Furthermore, for every dyadic interval $K_0$, we denote its left half by $K_0^\ell$ and its right
half by $K_0^r$.
Following the construction of Gamlen-Gaudet~\cite{gamlen:gaudet:1973}, we define the sets
\begin{equation*}
  B_{\widetilde I}^\ell
  = \bigcup\{ K_0^\ell : K_0\in \mathscr B_{\widetilde I} \}
  \quad\text{and}\quad
  B_{\widetilde I}^r
  = \bigcup \{ K_0^r : K_0\in \mathscr B_{\widetilde I} \}.
\end{equation*}
If $I$ is the \emph{left half} of $\widetilde I$ we put
\begin{subequations}\label{eq:intervals-case_1}
  \begin{equation}\label{eq:intervals-case_1:a}
    \mathscr F_m
    = \{K\in \dint : |K| = 2^{-m},\, K \subset B_{\widetilde I}^\ell \}.
  \end{equation}
  If $I$ is the \emph{right half} of $\widetilde I$ we define
  \begin{equation}\label{eq:intervals-case_1:b}
    \mathscr F_m
    = \{K\in \dint : |K| = 2^{-m},\, K \subset B_{\widetilde I}^r \}.
  \end{equation}
  See Figure~\ref{fig:gamlen-gaudet} for a depiction of $\mathscr F_m$.
  \begin{figure}[bt]
    \begin{center}
      \includegraphics{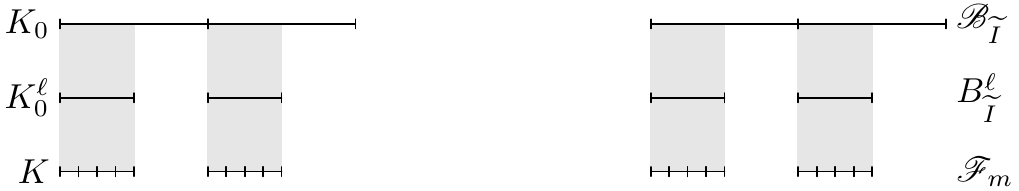}
    \end{center}
    \caption{The picture shows the construction of $\mathscr F_m$, if $I$ is the left half of
      $\widetilde I$.
      The large dyadic intervals $K_0$ on top form the set $\mathscr B_{\widetilde I}$.
      The medium sized dyadic intervals $K_0^\ell$ denote the left half of the $K_0$, and the
      set $B_{\widetilde I}^\ell$ is the union of the $K_0^\ell$.
      The small intervals $K$ at the bottom form the high-frequency cover $\mathscr F_m$ of the
      set $B_{\widetilde I}^\ell$.
    }
    \label{fig:gamlen-gaudet}
  \end{figure}
\end{subequations}
In either of the cases~\eqref{eq:intervals-case_1:a} and \eqref{eq:intervals-case_1:b} we put
\begin{equation}\label{eq:block-basis-candidate}
  f_m^{(\varepsilon)} = \sum_{K\in \mathscr F_m} \varepsilon_K h_K,
\end{equation}
for all $m\in \mathbb N$ and
$\varepsilon = (\varepsilon_K : \varepsilon_K \in \{\pm 1\}, K\in \mathscr F_m)$.

\subsubsection*{Choosing the frequency $m_i$}\hfill\\
\noindent
Note that by~\eqref{eq:intervals-case_1} and our induction hypothesis we have
$\mathscr F_m \cap \mathscr B_j = \emptyset$, $0\leq j \leq i-1$, $m\in \mathbb N$.
In particular, the sequence $(\varepsilon_K : \varepsilon_K\in \{\pm 1\},\ K\in \mathscr F_m)$ does
not interfere with any of the previous definitions of $b_j^{(\varepsilon)}$, $0 \leq j \leq i-1$.
By Lemma~\ref{lem:bracket-convergence}, we have that
\begin{equation*}
  \lim_{m\to \infty} \sup \big\{ |\langle f, f_m^{(\varepsilon)}\rangle| :
  \varepsilon = (\varepsilon_K : \varepsilon_K\in \{\pm 1\},\ K\in \mathscr F_m) \big\} = 0,
  \qquad f\in SL^\infty.
\end{equation*}
consequently, we obtain the estimate
\begin{equation}\label{eq:induction-step:a}
  \sup \Big\{
  \sum_{j=0}^{i-1} |\langle T b_j^{(\varepsilon)}, f_{m_i}^{(\varepsilon)}\rangle| :
  \varepsilon = (\varepsilon_K : \varepsilon_K\in \{\pm 1\},\ K\in \mathscr F_{m_i})
  \Big\}
  \leq \eta 4^{-i} \|f_{m_i}^{(\varepsilon)}\|_2^2,
\end{equation}
for sufficiently large ${m_i}$.
Certainly, we choose $m_i$ large enough so that $\mathscr F_{m_i}\neq \emptyset$,
see~\eqref{eq:intervals-case_1}.
Note that $\|f_{m_i}^{(\varepsilon)}\|_2^2 = |I|$.

\subsubsection*{Choosing the signs $\varepsilon$}\hfill\\
\noindent
Continuing with the proof, we obtain from~\eqref{eq:decomp} that
\begin{equation}\label{eq:block-basis-candidate:identity}
  T f_{m_i}^{(\varepsilon)}
  = \sum_{K\in \mathscr F_{m_i}} \varepsilon_K \alpha_K h_K + R_{m_i}^{(\varepsilon)},
\end{equation}
where
\begin{equation}\label{eq:block-basis-candidate:rest}
  R_{m_i}^{(\varepsilon)} = \sum_{K\in \mathscr F_{m_i}} \varepsilon_K r_K.
\end{equation}
For all $\varepsilon = (\varepsilon_K :  \varepsilon_K\in \{\pm 1\},\ K\in \mathscr F_{m_i})$, we
define
\begin{equation*}
  X_{m_i}(\varepsilon) = \langle R_{m_i}^{(\varepsilon)}, f_{m_i}^{(\varepsilon)} \rangle.
\end{equation*}
From~\eqref{eq:block-basis-candidate:identity} and~\eqref{eq:a-estimate} follows that
\begin{equation}\label{eq:diagonal-estimate:1}
  \langle T f_{m_i}^{(\varepsilon)}, f_{m_i}^{(\varepsilon)}\rangle
  \geq \delta \|f_{m_i}^{(\varepsilon)}\|_2^2 + X_{m_i}(\varepsilon).
\end{equation}
By~\eqref{eq:decomp} we have that $\langle r_K, h_K \rangle = 0$, hence
\begin{equation*}
  X_{m_i}(\varepsilon)
  = \sum_{\substack{K_0, K_1\in \mathscr F_{m_i}\\K_0\neq K_1}} \varepsilon_{K_0} \varepsilon_{K_1}
  \langle r_{K_0}, h_{K_1} \rangle.
\end{equation*}
Now, let $\cond_\varepsilon$ denote the averaging over all possible choices of signs
$\varepsilon = (\varepsilon_K : \varepsilon_K\in \{\pm 1\},\ K \in \mathscr F_{m_i})$.
If $K_0\neq K_1$, then $\cond_\varepsilon \varepsilon_{K_0} \varepsilon_{K_1} = 0$ and therefore
\begin{equation*}
  \cond_\varepsilon X_{m_i} = 0.
\end{equation*}
Taking the expectation $\cond_\varepsilon$ in~\eqref{eq:diagonal-estimate:1} and considering the
above identity, we obtain
\begin{equation}\label{eq:diagonal-estimate:2}
  \cond_\varepsilon \langle T f_{m_i}^{(\varepsilon)}, f_{m_i}^{(\varepsilon)}\rangle
  \geq \delta \|f_{m_i}^{(\varepsilon)}\|_2^2.
\end{equation}
The expectation on the right hand side is not present since
$\|f_{m_i}^{(\varepsilon)}\|_2^2 = |I|$ for all choices of $\varepsilon$.
By~\eqref{eq:diagonal-estimate:2} we can find an
$\varepsilon = (\varepsilon_K : \varepsilon_K\in \{\pm 1\},\ K\in \mathscr F_{m_i})$
such that
\begin{equation}\label{eq:diagonal_estimate}
  \langle T f_{m_i}^{(\varepsilon)}, f_{m_i}^{(\varepsilon)} \rangle
  \geq \delta \|f_{m_i}^{(\varepsilon)}\|_2^2.
\end{equation}

\subsubsection*{Choosing the set $\Lambda_{i+1}$}\hfill\\
\noindent
The next step is to find an infinite set $\Lambda_{i+1} \subset \Lambda_i\setminus\{m_i\}$ such
that~\eqref{eq:induction-properties:e} is satisfied.
To this end, we apply Lemma~\ref{lem:sieve}, to the infinite set $\Gamma$ given by
\begin{equation*}
  \Gamma = \{ n\in \Lambda_i : n > m_i \} \subset \Lambda_i.
\end{equation*}
Thus, we obtain $\Lambda_{i+1}\subset \Gamma$ such that
\begin{equation}\label{eq:small-future-estimate}
  \sup_{\|f\|_{SL^\infty}\leq 1} | \langle T P_{\Lambda_{i+1}} f, f_{m_i}^{(\varepsilon)}\rangle |
  \leq \eta 4^{-i} \|f_{m_i}^{(\varepsilon)}\|_2^2.
\end{equation}
We conclude the inductive construction step by defining
\begin{equation}\label{eq:induction-step:block_basis}
  \mathscr B_i = \mathscr B_I = \mathscr F_{m_i}
  \qquad\text{and}\qquad
  b_i^{(\varepsilon)} = b_I^{(\varepsilon)} = f_{m_i}^{(\varepsilon)}.
\end{equation}

\subsubsection*{Conclusion}\hfill\\
We remark that we chose ${m_i}$ and $\varepsilon$ according to~\eqref{eq:induction-step:a}
and~\eqref{eq:diagonal_estimate}, which together with~\eqref{eq:small-future-estimate} shows
Theorem~\ref{thm:quasi-diag}~\eqref{enu:thm:quasi-diag-ii}.  It follows immediately from the
Gamlen-Gaudet construction~\cite{gamlen:gaudet:1973} of $(\mathscr B_I : I\in \dint)$ that the collection
$(\mathscr B_I : I \in \dint)$ satisfies Jones' compatibility conditions~\textref[J]{enu:j1}--\textref[J]{enu:j4} with $\kappa_J=1$.\hfill\qedsymbol

\subsection{Factorization in \bm{$SL^\infty$} -- Proof of Theorem~\ref{thm:factor}}\label{subsec:proof:factor}\hfill\\
\noindent
We use the almost-diagonalization result in Section~\ref{subsec:diagonalization} to prove the main result
Theorem~\ref{thm:factor}.

Let $\delta, \eta > 0$, and let $T : SL^\infty\to SL^\infty$ be an operator satisfying
\begin{equation*}
  |\langle T h_I, h_I \rangle| \geq \delta |I|,
  \qquad I\in \dint.
\end{equation*}
Let $\eta' = \eta(\delta,\eta)$ denote a constant so that
\begin{equation}\label{eq:proof:factor:eta'}
  \frac{4\eta'}{\delta} < 1
  \qquad\text{and}\qquad
  \frac{1}{1 - \frac{4\eta'}{\delta}} \leq 1 + \eta.
\end{equation}
By Theorem~\ref{thm:quasi-diag}, we obtain a sequence of collections $(\mathscr B_I : I\in \dint)$
and a sequence of signs
$\varepsilon = (\varepsilon_K : \varepsilon_K \in \{\pm 1\},\ K\in \dint)$ which generate the block
basis of the Haar system $(b_I^{(\varepsilon)} : I\in \dint)$ given by
\begin{equation*}
  b_I^{(\varepsilon)} = \sum_{K \in \mathscr B_I} \varepsilon_K h_K,
  \qquad I\in \dint,
\end{equation*}
so that the following conditions are satisfied:
\begin{enumerate}[(i)]
\item $(\mathscr B_I : I\in \dint)$ satisfies Jones' compatibility
  conditions~\textref[J]{enu:j1}--\textref[J]{enu:j4} with constant $\kappa_J = 1$.
  \label{enu:proof:factor:quasi-diag-i}

\item For all $i\in \mathbb N_0$ we have the estimates
  \begin{subequations}\label{eq:proof:factor::quasi-diag-ii}
    \begin{align}
      \sum_{j=0}^{i-1} |\langle T b_j^{(\varepsilon)}, b_i^{(\varepsilon)}\rangle|
      & \leq \eta' 4^{-i} \|b_i^{(\varepsilon)}\|_2^2,
        \label{eq:proof:factor:quasi-diag-ii:a}\\
      \langle T b_i^{(\varepsilon)}, b_i^{(\varepsilon)} \rangle
      & \geq \delta\|b_i^{(\varepsilon)}\|_2^2,
        \label{eq:proof:factor:quasi-diag-ii:b}\\
      \sup \Big\{ |\langle T g , b_i^{(\varepsilon)}\rangle| :
      g = \sum_{j=i+1}^\infty a_j b_j^{(\varepsilon)},\ \|g\|_{SL^\infty}\leq 1
      \Big\}
      & \leq \eta' 4^{-i} \|b_i^{(\varepsilon)}\|_2^2.
        \label{eq:proof:factor:quasi-diag-ii:c}
    \end{align}
  \end{subequations}
  \label{enu:proof:factor:quasi-diag-ii}
\end{enumerate}
Since $(\mathscr B_I : I\in \dint)$ satisfies Jones' compatibility
conditions~\textref[J]{enu:j1}--\textref[J]{enu:j4} with $\kappa_J=1$, Remark~\ref{rem:projection} and
Theorem~\ref{thm:projection} imply that the operators
\begin{equation}\label{eq:B+Q:operators}
  B^{(\varepsilon)} f
  = \sum_{I\in \mathscr I} \frac{\langle f, h_I\rangle}{\|h_I\|_2^2} b_I^{(\varepsilon)}
  \qquad\text{and}\qquad
  Q^{(\varepsilon)} f
  = \sum_{I\in \mathscr I}
  \frac{\langle f, b_I^{(\varepsilon)}\rangle}{\|b_I^{(\varepsilon)}\|_2^2} h_I
\end{equation}
satisfy the estimates
\begin{equation}\label{eq:B+Q:estimates}
  \|B^{(\varepsilon)} f \|_{SL^\infty} \leq \|f\|_{SL^\infty}
  \quad\text{and}\quad
  \|Q^{(\varepsilon)} f \|_{SL^\infty} \leq \|f\|_{SL^\infty}.
\end{equation}
By~\eqref{eq:B+Q:estimates}, the operator $P^{(\varepsilon)} : SL^\infty\to SL^\infty$ defined as
$P^{(\varepsilon)} = B^{(\varepsilon)}Q^{(\varepsilon)}$, is given by
\begin{equation}\label{eq:P:definition}
  P^{(\varepsilon)} f
  = \sum_{I\in \mathscr I}
  \frac{\langle f, b_I^{(\varepsilon)}\rangle}{\|b_I^{(\varepsilon)}\|_2^2} b_I^{(\varepsilon)},
  \qquad f\in SL^\infty.
\end{equation}
Therefore, $P^{(\varepsilon)}$ is an orthogonal projection with the estimate
\begin{equation}\label{eq:P:estimate}
  \|P^{(\varepsilon)} f\|_{SL^\infty} \leq \|f\|_{SL^\infty},
  \qquad f\in SL^\infty.
\end{equation}

Let $Y$ denote the subspace of $SL^\infty$ given by
\begin{equation*}
  Y = \Big\{g = \sum_{i=0}^\infty a_i b_i^{(\varepsilon)} :
  a_i\in \mathbb R, \|g\|_{SL^\infty} < \infty\Big\}.
\end{equation*}
Note the following commutative diagram:
\begin{equation}\label{eq:commutative-diagram:preimage}
  \vcxymatrix{SL^\infty \ar[r]^\Id \ar[d]_{B^{(\varepsilon)}} & SL^\infty\\
    Y \ar[r]_\Id & Y \ar[u]_{{B^{(\varepsilon)}}^{-1}}}
  \qquad \|B^{(\varepsilon)}\|,\|{B^{(\varepsilon)}}^{-1}\| \leq 1.
\end{equation}
The estimates for $\|B^{(\varepsilon)}\|,\|{B^{(\varepsilon)}}^{-1}\|$ follow from~\eqref{eq:B+Q:estimates}.
Now, define $U : SL^\infty\to Y$ by
\begin{equation}\label{eq:almost-inverse}
  U f = \sum_{i=0}^\infty
  \frac{\langle f, b_i^{(\varepsilon)}\rangle}
  {\langle Tb_i^{(\varepsilon)}, b_i^{(\varepsilon)}\rangle}
  b_i^{(\varepsilon)},
\end{equation}
and note that by~\eqref{eq:induction-properties:d}, the $1$-unconditionality of the Haar system in
$SL^\infty$ and~\eqref{eq:P:estimate}, the operator $U$ has the upper bound
\begin{equation}\label{eq:almost-inverse:bounded}
  \|U : SL^\infty\to Y\|_{SL^\infty}
  \leq \frac{1}{\delta}.
\end{equation}
Observe that for all $g = \sum_{i=0}^\infty a_i b_i^{(\varepsilon)} \in Y$ the following identity
is true:
\begin{equation}\label{eq:crucial-identity}
  UTg - g
  = \sum_{i=0}^\infty \Big(
  \sum_{j : j < i} a_j
  \frac{\langle T b_j^{(\varepsilon)}, b_i^{(\varepsilon)}\rangle}
  {\langle Tb_i^{(\varepsilon)}, b_i^{(\varepsilon)}\rangle}
  + \frac{\big\langle T \sum_{j : j > i} a_j b_j^{(\varepsilon)}, b_i^{(\varepsilon)}\big\rangle}
  {\langle Tb_i^{(\varepsilon)}, b_i^{(\varepsilon)}\rangle}
  \Big)b_i^{(\varepsilon)}.
\end{equation}
Noting that $|a_j| \leq \|g\|_{SL^\infty}$ and using the estimates~\eqref{eq:induction-properties}
yields
\begin{equation}\label{eq:crucial-inequality}
  \|UTg - g\|_{SL^\infty} \leq \frac{4\eta'}{\delta} \|g\|_{SL^\infty}.
\end{equation}
Finally, let $J : Y\to SL^\infty$ denote the operator given by $Jy = y$.
By our choice~\eqref{eq:proof:factor:eta'} the operator
$V : SL^\infty\to Y$ given by $V=(UTJ)^{-1}U$ is well defined and
\begin{equation}\label{eq:commutative-diagram:image}
  \vcxymatrix{
    Y \ar[rr]^\Id \ar[dd]_J \ar[rd]_{UTJ} & & Y\\
    & Y \ar[ru]^{(UTJ)^{-1}} &\\
    SL^\infty \ar[rr]_T & & SL^\infty \ar[lu]_U \ar[uu]_V
  }
  \qquad \|J\|\|V\| \leq (1+\eta)/\delta.
\end{equation}
Merging the commutative diagrams~\eqref{eq:commutative-diagram:preimage}
and~\eqref{eq:commutative-diagram:image} concludes the proof.\hfill\qedsymbol

\begin{rem}
  We remark that in the identity~\eqref{eq:crucial-identity} above, the non-separability of
  $SL^\infty$ prevents us from expanding
  $\langle T \sum_{j : j > i} a_j b_j^{(\varepsilon)}, b_i^{(\varepsilon)}\rangle$
  into $\sum_{j : j > i} a_j \langle T b_j^{(\varepsilon)}, b_i^{(\varepsilon)}\rangle$.
  Therefore, in passing from~\eqref{eq:crucial-identity} to~\eqref{eq:crucial-inequality},
  we have to estimate the infinite sum
  $\langle T \sum_{j : j > i} a_j b_j^{(\varepsilon)}, b_i^{(\varepsilon)}\rangle$ \emph{directly}.
  This is achieved by Lemma~\ref{lem:sieve}, which results in
  estimate~\eqref{eq:induction-properties:e}.
\end{rem}

\subsection{\bm{$SL^\infty$} is primary -- Proof of Theorem~\ref{thm:primary}}\hfill\\
\noindent
Here we prove the second main result Theorem~\ref{thm:primary} and show that $SL^\infty$ is primary.
Since the following proof has been given in numerous situations, see e.g.~\cite{mueller:2005}
we will only describe its major steps.
\begin{itemize}
\item Diagonalization of $T$ by Theorem~\ref{thm:quasi-diag} with parameter $\delta = 0$ yields a block
  basis $(b_i : i\in \mathbb N_0)$ such that
  \begin{align*}
    \sum_{j=0}^{i-1} |\langle T b_i, b_j \rangle|
    & \leq \eta 4^{-i} \|b_i\|_2^2,\\
    \big|\big\langle T \sum_{j : j > i} a_j b_j, b_i\big\rangle\big|
    & \leq \eta 4^{-i} \|b_i\|_2^2\, \big\|\sum_{j : j > i} a_j b_j\big\|_{SL^\infty}.
  \end{align*}

\item Finding a ``large'' subcollection of dyadic intervals in one of the following collections:
  \begin{equation*}
    \{ I\in \dint : \langle T b_I, b_I\rangle \geq \|b_I\|_2^2/2\}
    \quad\text{or}\quad
    \{ I\in \dint : \langle (\Id - T) b_I, b_I\rangle \geq \|b_I\|_2^2/2\}.
  \end{equation*}
  It is well established how to construct a sequence of collections
  $(\mathscr C_I : I\in \dint)$ either entirely inside the first collection, or entirely inside the
  second collection, such that Jones' compatibility
  conditions~\textref[J]{enu:j1}--\textref[J]{enu:j4} are satisfied.
  We refer the reader to~\cite{gamlen:gaudet:1973}.
  See also~\cite{mueller:2005}.

\item Using the reiteration Theorem~\ref{thm:projection-iteration} and the projection
  Theorem~\ref{thm:projection} with parameter $\delta = 1/2$, we obtain a block basis
  $(c_I^{(\varepsilon)} : I\in \dint)$ of the Haar system given by
  \begin{equation*}
    c_I^{(\varepsilon)}
    = \sum_{K\in \mathscr C_I} b_K^{(\varepsilon)}
    = \sum_{K\in \mathscr C_I} \sum_{Q\in \mathscr B_K} \varepsilon_K h_K,
  \end{equation*}
  so that $(c_I^{(\varepsilon)} : I\in \dint)$ is $1$-equivalent to $(h_I : I\in \dint)$, and the
  subspace $Y$ of $SL^\infty$ defined by
  \begin{equation*}
    Y = \Big\{g = \sum_{I\in \dint} a_I c_i^{(\varepsilon)} :
    a_i\in \mathbb R, \|g\|_{SL^\infty} < \infty\Big\}.
  \end{equation*}
  is complemented in $SL^\infty$.
  The projection onto $Y$ can be chosen with norm $\leq 2+\eta$.

\item The rest of the proof is repeating the argument in the proof of Theorem~\ref{thm:projection}
  (with $c_I^{(\varepsilon)}$ taking the place of $b_I^{(\varepsilon)}$) to obtain that
  \begin{equation*}
    \vcxymatrix{SL^\infty \ar[r]^\Id \ar[d]_R & SL^\infty\\
      SL^\infty \ar[r]_T & SL^\infty \ar[u]_S}
    \qquad \|R\| \|S\| \leq 2 + \eta.
  \end{equation*}

\item Finally, consider the collections $\mathscr A_k\subset \dint$ given by
  \begin{equation*}
    \mathscr A_k
    = \{ I\in \dint : I\subset [1-2^{-k-1},1-2^{-k})\},
    \qquad k\in \mathbb N,
  \end{equation*}
  and note that with the obvious isomorphism we obtain that
  \begin{equation*}
    SL^\infty = \big( \sum SL^\infty \big)_\infty.
  \end{equation*}
  Thus, Pe{\l}czy{\'n}ski's decomposition method and the above factorization diagram imply the
  primarity of $SL^\infty$.\hfill\qedsymbol
\end{itemize}

\subsection*{Acknowledgments}\hfill\\
\noindent
It is my pleasure to thank N.J.~Laustsen and P.F.X.~Müller for many helpful discussions.
Supported by the Austrian Science Foundation (FWF) Pr.Nr. P28352.

\bibliographystyle{abbrv}
\bibliography{bibliography}

\end{document}